\documentclass[12pt]{article}

\usepackage{lineno,hyperref}

\usepackage{amsmath}
\usepackage{graphicx}
\usepackage{amsfonts}
\usepackage{verbatim}
\usepackage{latexsym}
\usepackage{color}

\makeatletter
\@addtoreset{equation}{section}

\begin{document}

\newcommand{\EE}{\mathbb{E}}
\newcommand{\PP}{\mathbb{P}}
\newcommand{\RR}{\mathbb{R}}
\newcommand{\SM}{\mathbb{S}}
\newcommand{\ZZ}{\mathbb{Z}}
\newcommand{\ind}{\mathbf{1}}
\newcommand{\LL}{\mathbb{L}}
\def\F{{\cal F}}
\def\G{{\cal G}}
\def\P{{\cal P}}

\newtheorem{theorem}{Theorem}[section]
\newtheorem{lemma}[theorem]{Lemma}
\newtheorem{coro}[theorem]{Corollary}
\newtheorem{defn}[theorem]{Definition}
\newtheorem{assp}[theorem]{Assumption}
\newtheorem{cond}[theorem]{Condition}
\newtheorem{expl}[theorem]{Example}
\newtheorem{prop}[theorem]{Proposition}
\newtheorem{rmk}[theorem]{Remark}
\newtheorem{conj}[theorem]{Conjecture}

\newcommand\tq{{\scriptstyle{3\over 4 }\scriptstyle}}
\newcommand\qua{{\scriptstyle{1\over 4 }\scriptstyle}}
\newcommand\hf{{\textstyle{1\over 2 }\displaystyle}}
\newcommand\hhf{{\scriptstyle{1\over 2 }\scriptstyle}}
\newcommand\hei{\tfrac{1}{8}}

\newcommand{\eproof}{\indent\vrule height6pt width4pt depth1pt\hfil\par\medbreak}

\def\a{\alpha}
\def\e{\varepsilon} \def\z{\zeta} \def\y{\eta} \def\o{\theta}
\def\vo{\vartheta} \def\k{\kappa} \def\l{\lambda} \def\m{\mu} \def\n{\nu}
\def\x{\xi}  \def\r{\rho} \def\s{\sigma}
\def\p{\phi} \def\f{\varphi}   \def\w{\omega}
\def\q{\surd} \def\i{\bot} \def\h{\forall} \def\j{\emptyset}

\def\be{\beta} \def\de{\delta} \def\up{\upsilon} \def\eq{\equiv}
\def\ve{\vee} \def\we{\wedge}

\def\D{\Delta} \def\O{\Theta} \def\L{\Lambda}
\def\X{\Xi} \def\Si{\Sigma} \def\W{\Omega}
\def\M{\partial} \def\N{\nabla} \def\Ex{\exists} \def\K{\times}
\def\V{\bigvee} \def\U{\bigwedge}

\def\1{\oslash} \def\2{\oplus} \def\3{\otimes} \def\4{\ominus}
\def\5{\circ} \def\6{\odot} \def\7{\backslash} \def\8{\infty}
\def\9{\bigcap} \def\0{\bigcup} \def\+{\pm} \def\-{\mp}
\def\la{\langle} \def\ra{\rangle}

\def\proof{\noindent{\it Proof. }}
\def\tl{\tilde}
\def\trace{\hbox{\rm trace}}
\def\diag{\hbox{\rm diag}}
\def\for{\quad\hbox{for }}
\def\refer{\hangindent=0.3in\hangafter=1}

\newcommand\wD{\widehat{\D}}
\newcommand{\ka}{\kappa_{10}}

\title{Invariant measures of the Milstein method for stochastic differential equations with commutative noise}

\author{Lihui Weng, Wei Liu\footnote{Corresponding author, Email: weiliu@shnu.edu.cn; lwbvb@hotmail.com}\\
Department of Mathematics, \\Shanghai Normal University, Shanghai, 200234, China
}
\date{}

\maketitle

\begin{abstract}
In this paper, the Milstein method is used to approximate invariant measures of stochastic differential equations with commutative noise. The decay rate of the transition probability kernel generated by the Milstein method to the unique invariant measure of the method is observed to be exponential with respect to the time variable. The convergence rate of the numerical invariant measure to the underlying counterpart is shown to be one. Numerical simulations are presented to demonstrate the theoretical results.
\end{abstract}

\section{Introduction}
In this paper, numerical approximations to invariant measures of stochastic differential equations (SDEs) of the It\^o type are studied. Briefly speaking, an unique invariant measure exists for the SDE
\par
\begin{equation*}
dx(t) = f(x(t))dt + g(x(t)) dB(t),
\end{equation*}
\par \noindent
when the drift and diffusion coefficients, $f$ and $g$, satisfy certain conditions.
\par
The explicit solutions to SDEs are hardly found, not to mention the invariant measures. Therefore, the approach that uses the numerical methods to SDEs has been attracting lots of attention recently. In this approach, one uses the numerical methods for SDEs to obtain the invariant measures of the numerical solutions and shows that the numerical one can converge to the underlying one when the step size tends to zero.
\par
For SDEs, authors in \cite{YM2004a} investigated the Euler-Maruyama method to approximate the invariant measure, the semi-implicit Euler method was studied in \cite{LM2015}, and the stochastic theta method was discussed in \cite{arXivJiangLiuWeng2018}. A modified truncated Euler-Maruyama method was considered in \cite{li2018explicit} for SDEs with both the drift and diffusion coefficients growing super-linearly.
\par
When the Markovian switching is combined with SDEs, authors in \cite{MYY2005a} and \cite{BSY2016} studied the Euler-Maruyama method. In a more recent paper \cite{LiMaYangYuan2018}, the authors investigated the backward Euler-Maruyama method.
\par
All the above works discussed the Euler-type methods, which provide the convergence rate, a half, of the numerical invariant measure to the underlying one (see Theorem 3.2 in \cite{BSY2016}).
\par
To our best knowledge, there has been few works on the approximation to invariant measures using the Milstein-type method. Therefore, we will investigate the ability of the Milstein method to approximate the invariant measures of a class of SDEs in this paper. Other higher order methods were also discussed, for example \cite{AVZ2014,SVZG2016,talay1990second}, we just mention some of them here and refer the readers to the references therein.
\par
The Milstein method was firstly proposed in \cite{Milstein1974}. Due to the higher convergence rate in the finite time, different kinds of the Milstein-type methods have been developed in recent years. The balanced Milstein method was introduced in \cite{KS2006}. A double-parameter Milstein method that can preserves positivity was constructed in \cite{HMS2013}. The tamed Milstein method was studied in \cite{WG2013}. The truncated Milstein method was investigated in \cite{GLMY2018}. For the detailed introduction to the Milstein method and other types of numerical methods for SDEs, we refer the readers to the monographs \cite{KP1992a,Mil02}.
\par
Although many works have been devoted to the Milstein-type methods, most of them focused on the finite time convergence or the asymptotic behaviour with zero as the attracting point. Few papers have devoted themselves to the asymptotic behaviour in terms of distributions. Therefore, this paper could also be regarded as a complement to the fruitful studies on the Milstein-type methods.
\par
The main contributions of this paper are threefold.
\begin{itemize}
\item Firstly, we observe that the transition probability kernels of the numerical solutions decay exponentially to the unique numerical invariant measure. Such a quick rate is important for finding the numerical invariant measure in relatively short time when the simulations are conducted.
\item Secondly, the numerical invariant measure generated by the Milstein method is shown to be convergent to the invariant measure of the underlying SDE with the rate one.
\item At last, numerical simulations are conducted, whose results are in line with the theoretical ones. To our best knowledge, no such simulation has been found in existing literatures. Therefore, the simulations in this paper present a methodology to check the theoretical results when such a topic is discussed.
\end{itemize}
This paper is constructed as follows. Section \ref{mathpre} contains the mathematical preliminaries. Main Results and their proofs are presented in Section \ref{mainresults}. Numerical simulations are displayed in Section \ref{secnumsim}. Section \ref{secconclu} concludes this paper.

\section{Mathematical Preliminaries} \label{mathpre}
In this paper, let $(\Omega , \F, \PP)$ be a complete probability space with a filtration $\left\{\F_t\right\}_{t \ge 0}$ satisfying the usual conditions that it is right continuous and increasing while $\F_0$ contains all $\PP$-null sets. Let $|\cdot|$ denote the Euclidean norm in $\RR^d$. Let $\langle \cdot,\cdot\rangle$ denote inner product in $\RR^d$. The transpose of a vector or matrix, $M$, is denoted by $M^T$ and the trace norm of a matrix, $M$, is denoted by $|M| = \sqrt{\trace(M^T M)}$. The maximum between $a$ and $b$ is denoted by $a \vee b$, and the minimum between $a$ and $b$ is denoted by $a \wedge b$. Denote the family of all probability measures on $\RR^d$ by $\P(\RR^d)$.
\par
For any $q\in(0,2]$, define a metric $d_q(\cdot,\cdot)$ on $\RR^d$ by
\par
\begin{equation*}
d_q(x,y) = |x - y|^q,~x,y\in \RR^d.
\end{equation*}
\par \noindent
For $q\in(0,2]$, the Wasserstein distance between $\mu \in \P(\RR^d)$ and $\mu' \in \P(\RR^d)$ is defined by
\par
\begin{equation*}
W_q (\mu, \mu') = \inf \EE \left( d_q(x,y) \right),
\end{equation*}
\par \noindent
where the infimum is taken over all pairs of random variables $x$ and $y$ on $\RR_d$ with respect to the laws $\mu$ and $\mu'$.
\noindent \par
Let $B(t)=(B^1(t),B^2(t),\ldots,B^m(t))^T$ be an $m$-dimensional Brownian motion. We consider the $d$-dimensional stochastic differential equation of the It\^o type
\par
\begin{equation}
 \label{SDE}
dx(t) = f(x(t))dt + \sum\limits_{j=1}^{m}g_{j}(x(t))dB^{j}(t)
\end{equation}
\par \noindent
with initial value $x(0) = x_0\in\mathbb{R}^{d}$, where $f:\mathbb{R}^{d}\rightarrow\mathbb{R}^d,~g_{j}:\mathbb{R}^d\rightarrow\mathbb{R}^d,~j=1,2,\ldots,m,$ and $x(t)=(x^{1}(t),x^{2}(t),\ldots,x^{d}(t))^{T}$.
\par
For the simplicity of the notations, we only consider the case of the commutative noise in this paper. Meanwhile, the case of the non-commutative noise is definitely interesting, but requires more careful analysis and more complicated notations. Due to the length of the paper, we will focus on the case of the commutative noise and report the more general case in the future work.
\par

In some of the proofs, we need the more specified notation that $g_j=(g_{1,j},g_{2,j},\ldots,g_{d,j})^T$, with $g_{i,j}:\mathbb{R}^d\rightarrow\mathbb{R}$ for $i=1,2,\ldots,d$ and $j=1,2,\ldots,m$. For $j_1,j_2=1,\ldots,m$, define
\par
\begin{equation*}
L^{j_1}g_{j_2}(x)=\sum\limits_{l=1}^{d}g_{l,j_1}(x)\frac{\partial g_{j_2}(x)}{\partial x^l}.
\end{equation*}
\par \noindent
Denote the transition probability kernel induced by the underlying solution, $x(t)$, by $\bar{\PP}_t(\cdot,\cdot)$, with the notation $\delta_z \bar{\PP}_t$ emphasizing the initial value $z$. Recall that a probability measure, $\pi(\cdot) \in \P(\RR^d)$, is called an invariant measure of $x(t)$, if
\par
\begin{equation*}
\pi(B) = \int_{\RR^d}  \bar{\PP}_t (x,B) \pi(dx)
\end{equation*}
\par \noindent
holds for any $t \geq 0$ and any Borel set $B \subset \RR^d$.

The following conditions are imposed on the drift and diffusion coefficients.
\begin{cond}
\label{ffgg}
 Assume that there exists a positive constant $\alpha$ such that for any $x,y \in \RR^d$
\begin{equation*}
 |f(x)-f(y)|^2\vee|g(x)-g(y)|^2\leq \alpha|x-y|^2.
\end{equation*}
\end{cond}
\begin{cond}
\label{xyff}
 Assume that there exists a positive constant $\sigma$ such that for any $x,y \in \RR^d$
\begin{equation*}
 2\langle x-y, f(x) - f(y)  \rangle+|g(x)-g(y)|^2 \leq -\sigma |x-y|^2.
\end{equation*}
\end{cond}
The next two conditions can be derived from Conditions \ref{ffgg} and \ref{xyff} but with a little bit complicated coefficients. For the simplicity, we give two new conditions as follows.
\begin{cond}
\label{ligrfg}
Assume that there exist positive constants $a$ and $b$ such that
for any $x\in\mathbb{R}^d$
\begin{equation*}
\label{fg}
 |f(x)|^2\vee|g(x)|^2 \leq a|x|^2+b.
\end{equation*}
\end{cond}

\begin{cond}
\label{xf}
Assume that there exist positive constants $\mu$ and $c$ such that for any $x\in \RR^d$
\begin{equation*}
 2\langle x, f(x) \rangle+|g(x)|^2 \leq -\mu|x|^2+c.
\end{equation*}
\end{cond}

\begin{cond}
\label{parg}
Assume that there exists a positive constant $\lambda$ such that for any $x\in \RR^d$, $j=1,2,\ldots,m$ and $l=1,2,\ldots,d$
\begin{equation*}
\left|\frac{\partial g_{j}(x)}{\partial x^l}\right|\leq \lambda.
\end{equation*}
\end{cond}
The Milstein method to the SDE (\ref{SDE}) is defined by
\par
\begin{equation}
\label{milm}
y_{k+1}=y_{k}+f(y_k)\Delta+\sum\limits_{j=1}^{m}g_j(y_k)\Delta B_k^j+\frac{1}{2}\sum\limits_{j_1=1}^{m}\sum\limits_{j_2=1}^{m}L^{j_1}g_{j_2}(y_k)\Delta B_k^{j_1}\Delta B_k^{j_2}-\frac{1}{2}\sum\limits_{j=1}^{m}L^jg_j(y_k)\Delta,
\end{equation}
\par \noindent
where $\Delta$ is the time step, $y_0=x(0)$, and $\Delta B_k^j$ is the Brownian motion increment in the $j$th component, $j=1,2,\ldots,m$.

For any $x \in \RR^d$ and any Borel set $B \subset \RR^d$, define the one-step and the $k$-step transition probability kernels for the numerical solutions, respectively, by
\begin{equation*}
 \PP(x,B) := \PP (y_1 \in B \big\vert y_0 = x)~~~\text{and}~~~\PP_k(x,B) := \PP (y_k \in B \big\vert y_0 = x).
\end{equation*}
If $\Pi_{\Delta} (\cdot) \in \P(\RR^d)$ satisfies
\begin{equation*}
\Pi_{\Delta} (B) = \int_{\RR^d} \PP_k(x,B) \Pi_{\Delta} (dx)
\end{equation*}
for any $t \geq 0$ and any Borel set $B \subset \RR^d$, then $\Pi_{\Delta} (\cdot)$ is called the numerical invariant measure of $y_k$.

\section{Main Results} \label{mainresults}
This section is divided into two parts. The first part sees the existence and uniqueness of the invariant measure of the Milstein method. The convergence of the numerical invariant measure to the underlying one is presented in the second part.

\subsection{The existence and uniqueness of the numerical invariant measure}

We firstly present our main theorem as follows and the proof is delayed to the end of this subsection.

\begin{theorem}
\label{mainthm}
Assume that Conditions \ref{ffgg} to \ref{parg} hold, then there exists a $\Delta^{\#} := \Delta^* \wedge \Delta^{**}$ such that for any given $\Delta  \in  (0,\Delta^{\#})$
the numerical solution generated by the Milstein method $\{y_k\}_{k \geq 0}$ has a unique invariant measure $\Pi_{\D}$.
\end{theorem}

To prove this theorem, we need two ingredients. Briefly speaking, the first one is the second moment boundedness of the numerical solution $y_k$ for $k = 0,1,2,\ldots$, and the second one is that two numerical solutions starting from two different initial values will get arbitrary close in the mean square sense when the time variable gets large.

\begin{lemma}
\label{lemma11}
 Assume Conditions \ref{ligrfg}, \ref{xf} and \ref{parg} hold, then there exists a $\Delta^* \in (0,1)$ such that for any $\Delta \in (0,\Delta^*)$ the solution generated by the Milstein method \eqref{milm} obeys
 \begin{equation*}
 \EE|y_{k}|^{2}\leq C_{1},~k=1,2,\ldots,
 \end{equation*}
 where $C_{1}$ is a constant that does not rely on k.
\end{lemma}
\begin{proof}
Taking squares and expectations on both sides of (\ref{milm}) and applying Conditions \ref{xf}, \ref{ligrfg} and \ref{parg}, we have
\begin{align*}
\mathbb{E}|y_{k+1}|^2=&\mathbb{E} \bigg (|y_{k}|^2+|f(y_k)|^2\Delta^2+|\sum\limits_{j=1}^{m}g_j(y_k)\Delta B_k^j|^2+\frac{1}{4} \large | \sum\limits_{j_1=1}^{m}\sum\limits_{j_2=1}^{m}L^{j_1}g_{j_2}(y_k)\Delta B_k^{j_1}\Delta B_k^{j_2}\\
&-\sum\limits_{j=1}^{m}L^jg_j(y_k)\Delta \large | ^2+2\langle y_{k},f(y_k)\Delta\rangle \bigg ) \\
\leq&\mathbb{E}|y_{k}|^2+\Delta^2(a\mathbb{E}|y_k|^2+b)+\Delta(-\mu\mathbb{E}|y_k|^2+c)\\
&+3\times2^{m^2-2}\lambda\Delta^2(a\mathbb{E}|y_k|^2+b)+2^{m-2}\lambda\Delta^2(a\mathbb{E}|y_k|^2+b)\\
=&A_1 \mathbb{E}|y_k|^2 + A_2,
\end{align*}
where the facts that $\mathbb{E}(\Delta B_k^j)=0$ for $j=1,2,\ldots,m$, and $\mathbb{E}(\Delta B_k^{j_1}\Delta B_k^{j_2})=\Delta$ if $j_1 = j_2$, $\mathbb{E}(\Delta B_k^{j_1}\Delta B_k^{j_2})=0$ if $j_1 \ne j_2$ for $j_1,j_2 = 1,2,\ldots,m$ are used. Let
\begin{equation*}
 A_1:=1-\mu\Delta +  (a+3\times2^{m^2-2}\lambda a+2^{m-2}\lambda a)\Delta^2,
\end{equation*}
and
\begin{equation*}
 A_2:= c\Delta + (b+3\times2^{m^2-2}\lambda b+2^{m-2}\lambda b)\Delta^2.
\end{equation*}
Due to $\mu > 0$, it is not hard to see that there exists a $\Delta^* \in (0,1)$ such that $A_1\in(0,1)$ and $A_2>0$ if $\Delta \in (0,\Delta^*)$. By iteration, we have
\begin{equation*}
\mathbb{E}|y_{k+1}|^2 \leq A_{1}\mathbb{E}|y_{k}|^2+A_{2}
\leq A_{1}^{k+1}\mathbb{E}|y_{0}|^2+A_2\frac{1}{1-A_1}\\
\leq \mathbb{E}|y_{0}|^2+A_2\frac{1}{1-A_1} := C_1.
\end{equation*}
This completes the proof.\eproof
\end{proof}

\begin{lemma}
\label{lemma12}
 Let Conditions \ref{ffgg} and \ref{xyff} hold. Then there exists a $\Delta^{**} \in (0,1)$ such that for any $\Delta \in(0,\Delta^{**})$ and any two initial values $x,y\in\mathbb{R}^{d}$ with $x \neq y$, the solutions generated by the milstein method \eqref{milm} satisfy
 \begin{equation*}
 \EE|y_{k}^x-y_{k}^y|^{2}\leq C_{2}\mathbb{E}|x-y|^2,
 \end{equation*}
where $C_{2}$ depends on $k$ with
\begin{equation}
\label{exponentialdecay}
\lim_{k \rightarrow + \infty} \frac{\log C_2}{k} <0.
\end{equation}
\end{lemma}
\begin{proof}
For the simplicity of the notations, we denote $y_k^x$ and $y_k^y$ by $x_k$ and $y_k$, respectively. From \eqref{milm}, we have
\begin{equation*}
\begin{split}
\label{xyk}
x_{k+1}-y_{k+1}=& x_{k}-y_{k}+(f(x_{k})-f(y_{k}))\Delta+\sum\limits_{j=1}^{m}(g_j(x_k)-g_j(y_k))\Delta B_k^j\\
&+\frac{1}{2}\sum\limits_{j_1=1}^{m}\sum\limits_{j_2=1}^{m}(L^{j_1}g_{j_2}(x_k)\Delta B_k^{j_1}\Delta B_k^{j_2}-L^{j_1}g_{j_2}(y_k)\Delta B_k^{j_1}\Delta B_k^{j_2})\\
&-\frac{1}{2}\sum\limits_{j=1}^{m}(L^jg_j(x_k)\Delta-L^jg_j(y_k)\Delta).
\end{split}
\end{equation*}
Taking squares and expectations on both sides and applying Conditions \ref{ffgg} and \ref{xyff}, in the similar manner as the proof of Lemma \ref{lemma11} we have
\begin{align*}
\mathbb{E}|x_{k+1}-y_{k+1}|^2
\leq&\mathbb{E}|x_k-y_{k}|^2+\Delta^2\alpha\mathbb{E}|x_k-y_k|^2-\sigma\Delta\mathbb{E}|x_k-y_k|^2\\
&+3\times2^{m^2-2}\lambda\Delta^2\alpha\mathbb{E}|x_k-y_k|^2+2^{m-2}\lambda\Delta^2\alpha\mathbb{E}|x_k-y_k|^2\\
=&A_3\mathbb{E}|x_k-y_k|^2,
\end{align*}
where
\begin{equation*}
A_3:=1-\sigma\Delta + (\alpha+3\times2^{m^2-2}\lambda \alpha+2^{m-2}\lambda \alpha)\Delta^2.
\end{equation*}
It is not hard see that there exists a $\Delta^{**} \in (0,1)$ such that $A_3\in(0,1)$ for any $\Delta \in(0,\Delta^{**})$. Then by iteration, we have
\begin{align*}
\mathbb{E}|x_{k+1}-y_{k+1}|^2\leq C_2\mathbb{E}|x-y|^2,
\end{align*}
where $C_2=A_3^{k+1}$. This completes the proof.\eproof
\end{proof}
\begin{rmk}
The inequality \eqref{exponentialdecay} in Lemma \ref{lemma12} also indicates that the transition probability kernel decays to the invariant measure in the exponential rate. The numerical simulation in Section \ref{secnumsim} demonstrates such an observation (see the left plot in Figure \ref{fig:num}).
\end{rmk}
Now we are ready to prove Theorem \ref{mainthm}.

\par \noindent
{\it Proof of Theorem \ref{mainthm}.} For each integer $m \geq 1$ and any Borel set $B \subset \RR^d$, define the measure
\begin{equation*}
\mu_m(B) = \frac{1}{m} \sum_{k=0}^m \PP(y_k \in B).
\end{equation*}
Lemma \ref{lemma11} together with the Chebyshev inequality yields that the measure sequence $\{\mu_m\}_{m\geq1}$ is tight. Then a subsequence that converges to an invariant measure can be extracted. This proves the existence of the numerical invariant measure.
\par
Assume $\Pi_{\Delta,1}$ and $\Pi_{\Delta,2}$ are two different invariant measure of $y_k^x$, then
\begin{equation*}
W_q(\Pi_{\Delta,1},\Pi_{\Delta,2}) = W_q(\Pi_{\Delta,1}\PP_k,\Pi_{\Delta,2}\PP_k)
\leq \int_{\RR^d} \int_{\RR^d}\Pi_{\Delta,1}(dx) \Pi_{\Delta,2}(dy) W_q(\delta_x \PP_k, \delta_y \PP_k).
\end{equation*}
From Lemma \ref{lemma12}, we have
\begin{equation*}
W_q(\delta_x \PP_k, \delta_y \PP_k)  \leq  \left( C_{2}\mathbb{E}|x-y|^2 \right)^{q/2} \rightarrow 0, ~\text{as}~k \rightarrow \infty.
\end{equation*}
Therefore, we have
\begin{equation*}
\lim_{k \rightarrow \infty}  W_q(\Pi_{\Delta,1},\Pi_{\Delta,2})= 0,
\end{equation*}
which indicates the uniqueness of the invariant measure. \eproof

\subsection{The Convergence of the numerical invariant measure to the underlying counterpart}
\label{theconvergencesec}

In this subsection, the main result on the convergence of the numerical invariant measure to the invariant measure of the underlying SDE is presented. We also reveal that by proper choosing the step size and the number of iterations of numerical solutions the convergence rate of one can be obtained.

\begin{theorem}
Given Conditions \ref{ffgg} to \ref{parg}, for any given $\Delta \in (0,\Delta^{\#})$ there exists a constant $C_3$ such that
\begin{equation*}
W_q (\pi, \Pi_\D) \leq C_3 \D^q,
\end{equation*}
where $q \in (0,2]$.
\end{theorem}
\begin{proof}
Note that for any $q \in (0,2]$
\begin{equation*}
W_q(\delta_x \bar{\PP}_{k\D},\pi) \leq \int_{\RR^d} \pi(dy) ~W_q(\delta_x \bar{\PP}_{k\D},\delta_y \bar{\PP}_{k\D}),
\end{equation*}
and
\begin{equation*}
W_q(\delta_x \PP_{k\D},\Pi_{\D}) \leq \int_{\RR^d} \Pi_{\D}(dy) ~W_q(\delta_x \PP_{k\D},\delta_y \PP_{k\D}).
\end{equation*}
Due to the existence and uniqueness of the invariant measure for the underlying SDE \eqref{SDE} \cite{BSY2016} and Theorem \ref{mainthm}, for the given $\Delta \in (0,\Delta^{\#})$, one can choose $k$ sufficiently large such that
\begin{equation*}
W_q(\delta_x \bar{\PP}_{k\D},\pi) \leq \frac{C_3}{3} \D^{q}~~~\text{and}~~~W_q(\delta_x \PP_{k\D},\Pi_{\D}) \leq \frac{C_3}{3} \D^{q}.
\end{equation*}
In addition, for the chosen $k$, it can be derived from \cite{Milstein1974} that
\begin{equation*}
W_q(\delta_x \bar{\PP}_{k\D},\delta_x \PP_{k\D}) \leq \frac{C_3}{3} \D^q.
\end{equation*}
Therefore, the proof is completed by the triangle inequality.   \eproof
\end{proof}

\section{Numerical simulations} \label{secnumsim}
In this Section, we use a scalar SDE as an example to demonstrate the following two facts.
\begin{itemize}
\item The numerical probability distribution (sometimes called the empirical distribution) converges to the invariant probability distribution in a quite fast speed for the Milstein method.
\item The numerical invariant probability distribution is convergent to the true invariant probability distribution with the rate of one.
\end{itemize}

\begin{expl}
We consider
\begin{equation}\label{eq:expl}
d x(t) = - 5 x(t)dt + (x(t) - 3)dB(t)
\end{equation}
with $x(0) = 2$.
\end{expl}
It is not hard to verify that there exists a unique invariant probability distribution for \eqref{eq:expl}. However, the explicit form of it is hardly found. Therefore, in this example we use the empirical distribution at time $t=5$ generated by the Milstein method as the true invariant probability distribution. More precisely, we simulate 1000 independent paths of \eqref{eq:expl} with $\Delta = 0.001$ from $t = 0$ to $t=5$, which gives us 1000 sample points at $t = 5$. Then the empirical distribution at $t = 5$ is constructed by that 1000 sample points. Here the statistic of the Kolmogorov-Smirnov test (K-S test) \cite{M1951a} is used to measure the difference  between two distributions
\par
Firstly, we test the change in differences between the numerical probability distributions and the true invariant probability distribution along the time line. Here the numerical probability distributions are generated by the Milstein method with $\Delta = 0.01$ and 1000 paths.
\par
The differences between the numerical probability distributions and the true invariant probability distribution along the time line are drew in the left plot of Figure \ref{fig:num}. It can be observed that as time advances the difference tend to zero, which indicates the numerical probability distribution indeed converges to the invariant probability distribution with a quite fast speed.
\par
Next we check the relation between the step size $\Delta$ and the difference between the numerical and true invariant probability distributions. We still regard the numerical probability distribution at large time $t=10$ with the small step size $\Delta = 2^{-5}$ as the true invariant probability distribution, where $100 \times 2^4$ paths are used. We also construct numerical invariant probability distributions with the following pairs of the step size and the number of paths, $(\Delta,m)$ that $(2^{-1}, 100)$, $(2^{-2}, 100 \times 2)$, $(2^{-3}, 100 \times 2^2)$ and $(2^{-4}, 100 \times 2^3)$.
\par
The loglog plot of the differences between the numerical and true invariant probability distributions against the step sizes are drew in the right picture of Figure \ref{fig:num} in green. The red line is the reference line with the slope of one. It can be seen that the convergence rate is approximately one.
\begin{rmk}
It should be mentioned that to get the convergence rate of one, we need to reduce the step size and increase the number of paths simultaneously. This is because that the difference between the numerical and true invariant probability distributions is determined by both errors of the numerical method for SDE and the Monte Carlo method to construct the empirical distribution. Therefore, one may not observe the correct convergence rate by just reducing the step size alone or increasing the number of paths alone.
\end{rmk}
\begin{figure}[ht]
\begin{center}$
\begin{array}{cc}
\includegraphics[width=2.5in]{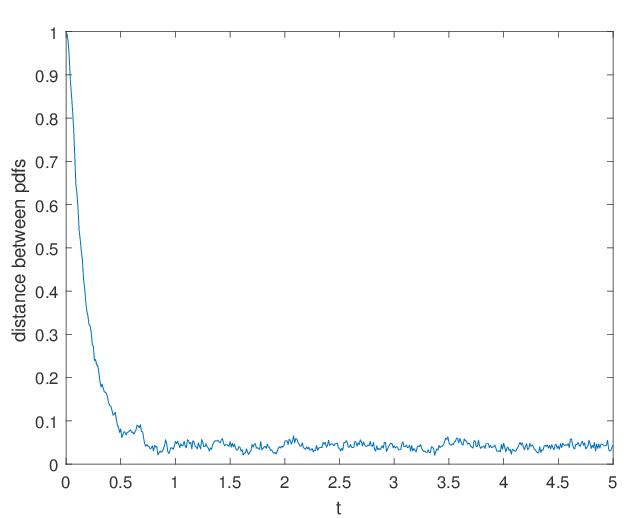}&
\includegraphics[width=2.5in]{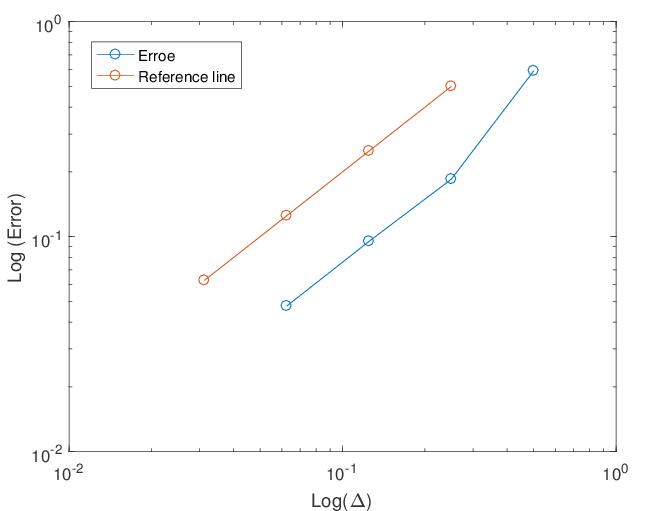}
\end{array}$
\end{center}
\caption{Left: the difference between the empirical and true distributions along the time line. Right: Loglog plot of errors against step sizes.}
\label{fig:num}
\end{figure}

\section{Conclusion} \label{secconclu}
In this paper, the Milstein method is proposed to approximate the invariant measure of a class of SDEs.
\par
The sufficient conditions on the drift and diffusion coefficients are provided to guarantee the convergence of the numerical invariant measure to the underlying one in the Wasserstein distance. The decay rate of the numerical transition probability kernel to the numerical invariant is observed to be exponential. The convergence rate of numerical invariant measure to the underlying invariant measure is shown to be one. Some numerical simulations are conducted to demonstrate the theoretical results.

\section*{Acknowledgement}
The authors would like to thank
the National Natural Science Foundation of China (11701378, 11871343),
``Chenguang Program'' supported by both Shanghai Education Development Foundation and Shanghai Municipal Education Commission (16CG50)
and
Shanghai Pujiang Program (16PJ1408000),
for their financial support.

\end{document}